\documentclass[10pt, reqno]{amsart}
\usepackage{amsmath, amsthm, amscd, amsfonts, amssymb, graphicx, color}
\usepackage[bookmarksnumbered, colorlinks, plainpages]{hyperref}
\hypersetup{colorlinks=true,linkcolor=red, anchorcolor=green, citecolor=cyan, urlcolor=red, filecolor=magenta, pdftoolbar=true}
\usepackage{mathrsfs}
\newtheorem{theorem}{Theorem}[section]
\newtheorem{lemma}[theorem]{Lemma}

\theoremstyle{definition}
\newtheorem{definition}[theorem]{Definition}
\newtheorem{example}[theorem]{Example}

\theoremstyle{remark}
\newtheorem{remark}[theorem]{Remark}
\numberwithin{equation}{section}

\begin{document}
\setcounter{page}{1}

\title[Geometry of integers revisited]{Geometry of integers revisited}

\author[Nikolaev]
{Igor V.  Nikolaev$^1$}

\address{$^{1}$ Department of Mathematics and Computer Science, St.~John's University, 8000 Utopia Parkway,  
New York,  NY 11439, United States.}
\email{\textcolor[rgb]{0.00,0.00,0.84}{igor.v.nikolaev@gmail.com}}

%\dedicatory{In memory of Ola Bratteli}

\subjclass[2010]{Primary 11R04, 14H55; Secondary 46L85.}

\keywords{arithmetic geometry, Serre $C^*$-algebra.}

%\date{Received:  August 14, 2015; Revised: yyyyyy; Accepted: zzzzzz.}

\begin{abstract}
We study   geometry of  the ring of integers $O_K$ of a number field $K$. 
Namely, it is proved that the inclusion $\mathbf{Z}\subset O_K$ defines a covering 
of the Riemann sphere $\mathbf{C}P^1$ ramified over the points $\{0,1,\infty\}$.  
 Our approach is based on the notion of a Serre $C^*$-algebra.   
 As an application, a new   proof of the Belyi Theorem
 is given. 
 \end{abstract}

\maketitle

%**************************************************************************
\section{Introduction}
%***************************************************************************
An interplay  between arithmetic and geometry is well known  [Weil 1949] \cite{Wei1}.  
The  Weil's Conjectures were a  motivation  for the notion  of a scheme
[Grothendieck 1960]  \cite{Gro1}.  Recall that the spectrum $Spec~R$  
of  a commutative ring $R$ is  the set of all prime ideals of $R$ endowed 
with the Zariski topology.   Such a  topology  is non-Hausdorff but admits a 
cohomology theory and an analog of the Lefschetz Fixed-Point Theorem.
The latter is enough to prove Weil's Conjectures. 

Let  $\mathbf{Z}$ be  the ring of integers.  It was noticed long ago  that the  space $Spec~\mathbf{Z}$
 is ``similar'' to   the Riemann sphere $\mathbf{C}P^1$  [Eisenbud \& Harris 1999] \cite[p. 83]{EH}.
Moreover,   if  $O_K$ is the ring of integers 
of a number field $K$,  then  the inclusion $\mathbf{Z}\subset O_K$
corresponds to  a Riemann surface $\mathscr{R}$,  such that 
there exists a ramified covering map
$\mathscr{R}\to\mathbf{C}P^1$. 
The  Grothendieck's  theory of schemes  cannot explain  this analogy 
 [Manin 2006] \cite[Section 2.2]{Man1}.

\medskip
In this note we clarify the relation  between  the ring  $\mathbf{Z}$ and 
the sphere $\mathbf{C}P^1$. 
Namely, it is proved   that  the inclusion  $\mathbf{Z}\subset O_K$ defines  
a  covering $\mathscr{R}\to\mathbf{C}P^1$ ramified over three points $\{0,1,\infty\}$ (theorem \ref{thm1.3}).
Our approach  is based on the notion of a Serre $C^*$-algebra \cite[Section 5.3.1] {N}.   
 To formalize our results, we  need the following definitions. 

\medskip
Let $V$ be a complex projective variety. Denote by $B(V, \mathcal{L}, \sigma)$ the twisted 
homogeneous coordinate ring of $V$, where $\mathcal{L}$ is an invertible sheaf and $\sigma$
is an automorphism of $V$ [Stafford \& van ~den ~Bergh 2001]  \cite[p. 173]{StaVdb1}. 
Recall that the  Serre $C^*$-algebra, $\mathscr{A}_V$,   is the norm closure of a
self-adjoint representation of the ring  $B(V, \mathcal{L}, \sigma)$ by the bounded linear 
operators  on a Hilbert space $\mathscr{H}$;  such an algebra depends on $V$ alone,
since the values of $\mathcal{L}$ and $\sigma$ are fixed by the $\ast$-involution of algebra
$B(V, \mathcal{L}, \sigma)$  \cite[Section 5.3.1] {N}. 
The map $V\mapsto \mathscr{A}_V$ is a functor. 
 Namely, if $V$ and $V'$ are  defined over a number field $K\subset\mathbf{C}$,
then $V$ is $K$-isomorphic to $V'$ if and only if the algebra $\mathscr{A}_V$ is isomorphic to  $\mathscr{A}_{V'}$. 
In  contrast, the variety $V$ is $\mathbf{C}$-isomorphic to $V'$ if and only if  $\mathscr{A}_V$ is Morita equivalent to  
$\mathscr{A}_{V'}$, i.e.  $\mathscr{A}_V\otimes\mathscr{K}\cong \mathscr{A}_{V'}\otimes\mathscr{K}$,
where  $\mathscr{K}$  is the $C^*$-algebra of  compact operators \cite[Corollary 1.2]{Nik1}.
In other words, the tensor  product  $\mathscr{A}_V\otimes\mathscr{K}$ is  an analog of the change of base  from $K$
to $\mathbf{C}$.

The latter remark
can be used to ``geometrize''  the ring $O_K$ as follows. 
Recall that there exists an isomorphism  $B(V, \mathcal{L}, \sigma)\cong M_2(R)$, 
where $R$ is the  homogeneous coordinate ring of a variety $V$ 
[Stafford \& van ~den ~Bergh 2001]  \cite[Section 8]{StaVdb1}. 
If $R\cong O_K$,  then the norm closure of a self-adjoint representation of  the ring 
$M_2(O_K)$ is  a $C^*$-algebra which we denote by  $\mathscr{A}_{O_K}$.
 Notice that in general  the  $\mathscr{A}_{O_K}$ is no longer  the  Serre $C^*$-algebra. 
However, changing the base from $K$ to $\mathbf{C}$,  we conclude that  
  the  tensor  product $\mathscr{A}_{O_K}\otimes\mathscr{K}$ must  be isomorphic  
to a Serre  $C^*$-algebra.   Thus,  one gets  the following definition. 
%*********************************************************************************************
\begin{definition}
The complex projective variety $V$ will be called an {\it avatar}
%****************************************************************************************
\footnote{For the lack of a better word meaning the ``image''.} 
%****************************************************************************************
 of the ring $O_K$, if there exists  a $C^*$-algebra homomorphism 
 %***************************************************************************************
 \begin{equation}
 h: \mathscr{A}_V\to\mathscr{A}_{O_K}\otimes\mathscr{K}. 
 \end{equation}
 %****************************************************************************************
\end{definition}
%**********************************************************************************************
%**********************************************************************************************
\begin{example}
If $R$ is the homogeneous coordinate ring of a complex projective variety $V$, then 
$V$ is the avatar of  $R$.  In this case,   $\mathscr{A}_{R}\otimes\mathscr{K}\cong\mathscr{A}_V$,
i.e. the map $h$ is a $C^*$-algebra isomorphism. 
\end{example}
%********************************************************************************************
Our main result can be formulated as follows. 
%**********************************************************************************************
\begin{theorem}\label{thm1.3}
Let $\mathbf{Z}$ be the ring of rational integers and let $O_K$ be the ring of algebraic 
integers of a number field $K$.   Then:

\medskip
(i) the Riemann sphere $\mathbf{C}P^1$ is an avatar of the ring $\mathbf{Z}$;

\smallskip
(ii)  there exists a Riemann surface $\mathscr{R}=\mathscr{R}(K)$,  
such that $\mathscr{R}$ is an avatar of the  ring  $O_K$;

\smallskip
(iii)  the inclusion $\mathbf{Z}\subset O_K$ defines a covering $\mathscr{R}\to \mathbf{C}P^1$
ramified over the  points $\{0,1, \infty\}$. 
\end{theorem}
%************************************************************************************
The article is organized as follows. In Section 2 we briefly review 
noncommutative algebraic geometry and arithmetic groups.  
Theorem \ref{thm1.3} is proved  in Section 3. 
As an application of theorem \ref{thm1.3},  we give a new proof of the Belyi Theorem 
 [Belyi 1979] \cite[Theorem 4]{Bel1}.

%**************************************************************************
\section{Preliminaries}
%***************************************************************************
We  review some facts of noncommutative algebraic geometry and 
 arithmetic groups. The reader is referred to [Humphreys 1980] \cite{H} and 
[Stafford \& van ~den ~Bergh 2001]  \cite{StaVdb1} for a detailed account.

%**************************************************************************
\subsection{Noncommutative algebraic geometry}
%***************************************************************************
Let $V$ be a projective variety over the field $k$.  Denote by $\mathcal{L}$ an invertible
sheaf of the linear forms on $V$.  If $\sigma$ is an automorphism of $V$,  then
the pullback of $\mathcal{L}$ along $\sigma$ will be denoted by $\mathcal{L}^{\sigma}$,
i.e. $\mathcal{L}^{\sigma}(U):= \mathcal{L}(\sigma U)$ for every $U\subset V$. 
The graded $k$-algebra
 %*************************************************************************
\begin{equation}\label{eq2.1}
B(V, \mathcal{L}, \sigma)=\bigoplus_{i\ge 0} H^0\left(V, ~\mathcal{L}\otimes \mathcal{L}^{\sigma}\otimes\dots
\otimes  \mathcal{L}^{\sigma^{ i-1}}\right)
\end{equation}
%*************************************************************************  
is called a {\it twisted homogeneous coordinate ring} of $V$.  Such a ring is 
always non-commutative,  unless the automorphism $\sigma$ is trivial. 
A multiplication of sections of  $B(V, \mathcal{L}, \sigma)=\oplus_{i=1}^{\infty} B_i$ is defined by the 
rule  $ab=a\otimes b$,   where $a\in B_m$ and $b\in B_n$.
An invertible sheaf $\mathcal{L}$ on $V$  is called $\sigma$-ample, if for 
every coherent sheaf $\mathcal{F}$ on $V$,
 the cohomology group $H^k(V, ~\mathcal{L}\otimes \mathcal{L}^{\sigma}\otimes\dots
\otimes  \mathcal{L}^{\sigma^{ n-1}}\otimes \mathcal{F})$  vanishes for $k>0$ and
$n>>0$.   If $\mathcal{L}$ is a $\sigma$-ample invertible sheaf on $V$,  then
%************************************************************************************
\begin{equation}\label{eq2.2}
Mod~(B(V, \mathcal{L}, \sigma)) / ~Tors ~\cong ~Coh~(V),
\end{equation}
%****************************************************************************
where  $Mod$ is the category of graded left modules over the ring $B(V, \mathcal{L}, \sigma)$,
$Tors$ is the full subcategory of $Mod$ of the torsion  modules and  $Coh$ is the category of 
quasi-coherent sheaves on a scheme $V$.  In other words, the $B(V, \mathcal{L}, \sigma)$  is  
a coordinate ring of the variety $V$.
%******************************************************************
\begin{example}\label{ex2.1}
([Stafford \& van den Bergh 2001]  \cite[p.173]{StaVdb1})
Denote by $P^1(k)$  a projective line over the field $k$.
Consider an automorphism $\sigma$ of the $P^1(k)$
given by the formula $\sigma(u)=qu$, where $u\in P^1(k)$
and $q\in k^{\times}$.  Then  $B(P^1(k), \mathcal{L}, \sigma)\cong U_q$,
where   $U_q$  is  the $k$-algebra of polynomials in 
variables $x_1$ and $x_2$ satisfying  a commutation relation:
%**********************************************************************************************
\begin{equation}\label{eq2.3}
x_2x_1=qx_1x_2.
\end{equation}
%**********************************************************************************************
\end{example}
%***********************************************************************
%******************************************************************
\begin{example}
([Stafford \& van den Bergh 2001]  \cite[p.197]{StaVdb1})
Denote by 
$\mathcal{E}(k)=\{(u,v,w,z)\in P^3(k) ~|~u^2+v^2+w^2+z^2 =
{1-\alpha\over 1+\beta}v^2+{1+\alpha\over 1-\gamma}w^2+z^2 = 0\}$  
an elliptic curve  over the field $k$,  where $\alpha,\beta,\gamma\in k$
are constants, such that  $\beta\ne -1$ and  $\gamma\ne 1$.
Let  $\sigma$  be a shift automorphism of the $\mathcal{E}(k)$.
Then  $B(\mathcal{E}(k), \mathcal{L}, \sigma)\cong S(\alpha,\beta,\gamma)$,
where   $S(\alpha,\beta,\gamma)$  is  the  Sklyanin algebra on four  generators
$x_i$ satisfying  the commutation relations:
 %***********************************************************************************************
\begin{equation}\label{eq2.4}
\left\{
\begin{array}{ccc}
x_1x_2-x_2x_1 &=& \alpha(x_3x_4+x_4x_3),\\
x_1x_2+x_2x_1 &=& x_3x_4-x_4x_3,\\
x_1x_3-x_3x_1 &=& \beta(x_4x_2+x_2x_4),\\
x_1x_3+x_3x_1 &=& x_4x_2-x_2x_4,\\
x_1x_4-x_4x_1 &=& \gamma(x_2x_3+x_3x_2),\\ 
x_1x_4+x_4x_1 &=& x_2x_3-x_3x_2,
\end{array}
\right.
\end{equation}
%*****************************************************************************
where $\alpha+\beta+\gamma+\alpha\beta\gamma=0$.  
\end{example}
%***********************************************************************
%******************************************************************
\begin{example}\label{ex2.3}
(\cite[Lemma 3.1]{Nik2})
Let $\mathscr{R}$ be an arithmetic Riemann surface, i.e. given by the 
AF-algebra  of stationary type \cite[Section 5.2]{N}. 
(Such Riemann surfaces can be identified with the complex 
algebraic curves  defined over a number field.).  Then 
%***************************************************************************************
\begin{equation}
B(\mathscr{R}, \mathcal{L}, \sigma)\cong R[\pi_1(S^3 \backslash \mathscr{L})],
\end{equation}
%**************************************************************************************
where  $\mathscr{L}$ is a link embedded in the three-sphere $S^3$ 
and  $R[\pi_1(S^3 \backslash \mathscr{L})]$
is the group ring of the fundamental group $\pi_1(S^3 \backslash \mathscr{L})$. 
\end{example}
%***********************************************************************

%**************************************************************************
\subsection{Arithmetic groups}
%***************************************************************************
Let $G$ be a linear algebraic group defined over the field $\mathbf{Q}$.
Denote by $G_{\mathbf{Z}}$ the group of integer points of $G$. A subgroup 
$\Gamma\subset G$ is called {\it arithmetic} if $\Gamma$ is commensurable with 
the  $G_{\mathbf{Z}}$, i.e. $\Gamma\cap G_{\mathbf{Z}}$ has a finite index both
in $\Gamma$ and $G_{\mathbf{Z}}$. Informally, the arithmetic group is a discrete 
subgroup of the group $GL_n(\mathbf{C})$ defined by some arithmetic properties.
For instance, $\mathbf{Z}\subset\mathbf{R}$, $GL_n(\mathbf{Z})\subset GL_n(\mathbf{R})$ and
$SL_n(\mathbf{Z})\subset SL_n(\mathbf{R})$ are  examples of the arithmetic groups.

Denote by $\mathcal{O}$ the ring of algebraic integers of all finite  
extensions of the number field $\mathbf{Q}$. 
Let $\mathbb{H}^3$ be the hyperbolic 3-dimensional space.
The following remarkable result establishes 
a deep link between arithmetic groups and  topology.
%*******************************************************************************
\begin{theorem}\label{thm2.4}
{\bf ([Maclachlan \& Reid  2003] \cite[p. 169]{MR})}
Let $M=\mathbb{H}^3/\Gamma$ be a finite volume hyperbolic  3-manifold.
Then $\Gamma$ is conjugate to a subgroup of the group $PSL_2(\mathcal{O})$.  
\end{theorem}
%**************************************************************************** 
%******************************************************************
\begin{example}\label{ex2.5}
Let $\mathscr{L}$ be a hyperbolic link, i.e.  $S^3\backslash\mathscr{L}\cong \mathbb{H}^3/\Gamma$
for an arithmetic group  $\Gamma$. Then 
%**********************************************************************
\begin{equation}\label{eq2.5} 
\pi_1(S^3 \backslash\mathscr{L})\cong\Gamma.
\end{equation}
%*****************************************************************
\end{example}
%***********************************************************************
%***********************************************************************
%\begin{remark}\label{rmk2.6}
%By Thurston's classification, every link $\mathscr{L}$ is hyperbolic,  unless $\mathscr{L}$ 
%is a torus or a satellite knot. Moreover almost all  links are hyperbolic, i.e. a randomly taken
% prime link is always hyperbolic.    In view of Theorem \ref{thm2.4} and Example \ref{ex2.5},
%one gets a family, $\mathscr{F}$,  of number fields $K$ corresponding to the hyperbolic links. 
%In what follows, we restrict to $\mathscr{F}$, such that each $K\in\mathscr{F}$ is 
%a Galois extension of $\mathbf{Q}$.  
%\end{remark}
%*******************************************************************

%**************************************************************************
\section{Proof of theorem \ref{thm1.3}}
%***************************************************************************
(i)  Let us show that  the  $\mathbf{C}P^1$ is an avatar of  $\mathbf{Z}$.
Indeed, in this case $R\cong\mathbf{Z}$ and $\mathscr{A}_{\mathbf{Z}}$ is the closure of a self-adjoint representation of
the ring $M_2(\mathbf{Z})$. Consider the group $PSL_2(\mathbf{Z})=SL_2(\mathbf{Z})/\pm I$, where  $SL_2(\mathbf{Z})$ is the group 
of invertible  elements of $M_2(\mathbf{Z})$. 
Recall that the group $PSL_2(\mathbf{Z})$ is  generated by the matrices: 
%********************************************************************************
\begin{equation}\label{eq3.1}
u=\left(
\begin{matrix}
0 & 1\cr -1 & 0
\end{matrix}
\right)
\quad
\hbox{and}
   \quad v=\left(
\begin{matrix}
0 & -1\cr 1 & -1
\end{matrix}
\right)
\end{equation}
%******************************************************************************* 
which satisfy the relations modulo $\pm I$:
%********************************************************************************
\begin{equation}\label{eq3.2}
u^2= v^3=1.
\end{equation}
%******************************************************************************* 

\bigskip
On the other hand, consider Example \ref{ex2.1} with $k\cong\mathbf{Q}$ and assume  
that $q=-1$ in relation (\ref{eq2.3}).  In other words, one gets a relation:
%********************************************************************************
\begin{equation}\label{eq3.3}
x_2x_1=-x_1x_2.
\end{equation}
%******************************************************************************* 

Consider a substitution:
%***********************************************************************************************
\begin{equation}\label{eq3.4}
\left\{
\begin{array}{ccl}
u&=& x_2x_1x_2^{-1}x_1^{-1}\\
v &=& x_2. 
\end{array}
\right.
\end{equation}
%*****************************************************************************

The reader can verify, that substitution (\ref{eq3.4}) and relation (\ref{eq3.3}) 
reduces relations (\ref{eq3.2}) to the form: 
%********************************************************************************
\begin{equation}\label{eq3.5}
x_2^3=1.
\end{equation}
%******************************************************************************* 

Let $\mathscr{I}$ be a two-sided ideal in the algebra $B(P^1(\mathbf{Q}), \mathcal{L}, \sigma)$ of 
Example \ref{ex2.1} generated by relation (\ref{eq3.5}).  In view of (\ref{eq3.2})-(\ref{eq3.5}), one gets a
ring  isomorphism: 
%********************************************************************************
\begin{equation}\label{eq3.6}
B(P^1(\mathbf{Q}), \mathcal{L}, \sigma)/\mathscr{I}\cong M_2(\mathbf{Z}). 
\end{equation}
%******************************************************************************* 

Let $\rho$ be a self-adjoint representation of the ring $B(P^1(\mathbf{Q}), \mathcal{L}, \sigma)$
by the linear operators on a Hilbert space $\mathscr{H}$. Notice that such a representation 
exists, because relation (\ref{eq3.3}) is invariant under the involution $x_1^*=x_2$ and 
$x_2^*=x_1$. Since $\rho(B(P^1(\mathbf{Q}), \mathcal{L}, \sigma))=\mathscr{A}_{P^1(\mathbf{Q})}$
and $\rho(M_2(\mathbf{Z}))=\mathscr{A}_{\mathbf{Z}}$, it follows from (\ref{eq3.6}) that there exists a 
$C^*$-algebra homomorphism
%********************************************************************************
\begin{equation}\label{eq3.7}
h:  \mathscr{A}_{P^1(\mathbf{Q})}\to\mathscr{A}_{\mathbf{Z}},
\end{equation}
%*******************************************************************************  
where  $Ker ~h=\rho(\mathscr{I})$.  
The homomorphism $h$ extends to a homomorphism between the 
 products
%********************************************************************************
\begin{equation}\label{eq3.8}
h:  \mathscr{A}_{P^1(\mathbf{Q})}\otimes\mathscr{K}\to\mathscr{A}_{\mathbf{Z}}\otimes\mathscr{K}, 
\end{equation}
%*******************************************************************************  
where $\mathscr{K}$ is the $C^*$-algebra of compact operators. 
But $\mathscr{A}_{P^1(\mathbf{Q})}\otimes\mathscr{K}\cong \mathscr{A}_{\mathbf{C}P^1}$
and, therefore, one gets a $C^*$-algebra homomorphism
%********************************************************************************
\begin{equation}\label{eq3.9}
h: \mathscr{A}_{\mathbf{C}P^1} \to\mathscr{A}_{\mathbf{Z}}\otimes\mathscr{K}. 
\end{equation}
%*******************************************************************************  

In other words, the Riemann sphere $\mathbf{C}P^1$ is an avatar of the ring $\mathbf{Z}$.

\bigskip
(ii)  Let us show that if $K$ is a number field, then there exists a Riemann surface $\mathscr{R}$, 
such that $\mathscr{R}$ is an avatar of the ring $O_K$. 
Indeed,   we can always assume that $K$ has at least one complex embedding and fix one of such embeddings 
$K\not\subset\mathbf{R}$.  
 (For otherwise, we replace $K$ by a CM-field of $K$, i.e. a totally imaginary quadratic extension
of the totally real field $K$.  This case corresponds to the double covering $\mathscr{R}'$of the Riemann surface  $\mathscr{R}$.)  
For simplicity, let $R\cong O_K$ and $\Gamma\cong PSL_2(O_K)$.  (The case of a non-maximal order $\Lambda\subseteq O_K$ is treated likewise and corresponds to the covering of 
the Riemann surface $\mathscr{R}$.)  In view of  (\ref{eq2.5}),   there exists a hyperbolic link $\mathscr{L}$,  
such that: 
%**********************************************************************
\begin{equation}\label{eq3.10} 
PSL_2(O_K)\cong\pi_1(S^3 \backslash\mathscr{L}).
\end{equation}
%*****************************************************************

\bigskip
On the other hand,  it is known that
%**********************************************************************
\begin{equation}\label{eq3.11} 
R[\pi_1(S^3 \backslash\mathscr{L})]\cong B(\mathscr{R}, \mathcal{L}, \sigma), 
\end{equation}
%*****************************************************************
where $R[\pi_1(S^3 \backslash\mathscr{L})]$ is the group ring of $\pi_1(S^3 \backslash\mathscr{L})$ and 
 $\mathscr{R}$ is a Riemann surface, see example \ref{ex2.3}. In particular, it follows from (\ref{eq3.10}) 
 that  
%**********************************************************************
\begin{equation}\label{eq3.12} 
B(\mathscr{R}, \mathcal{L}, \sigma)\cong R[PSL_2(O_K)]. 
\end{equation}
%*****************************************************************

Let $\rho$ be a self-adjoint representation of the ring $B(\mathscr{R}, \mathcal{L}, \sigma)$
by the linear operators on a Hilbert space $\mathscr{H}$. 
The norm closure of $\rho(B(\mathscr{R}, \mathcal{L}, \sigma))$
is the Serre $C^*$-algebra $\mathscr{A}_{\mathscr{R}}$.

On the other hand, it follows from (\ref{eq3.12}) that  taking the norm closure of 
$\rho(R[PSL_2(O_K)])$,  one gets a $C^*$-algebra $\mathscr{A}_{O_K}$, 
such that 
%**********************************************************************
\begin{equation}\label{eq3.13} 
\mathscr{A}_{O_K}\otimes\mathscr{K}\cong \mathscr{A}_{\mathscr{R}}. 
\end{equation}
%*****************************************************************

In other words,  there exists an isomorphism:
%**********************************************************************
\begin{equation}\label{eq3.14} 
h: \mathscr{A}_{\mathscr{R}}\to \mathscr{A}_{O_K}\otimes\mathscr{K}. 
\end{equation}
%*****************************************************************

It follows from (\ref{eq3.14}) that  the Riemann surface $\mathscr{R}$ is an avatar of 
the ring $O_K$.

\bigskip
(iii) Finally, let us show that 
 the inclusion $\mathbf{Z}\subset O_K$ defines a covering $\mathscr{R}\to \mathbf{C}P^1$
ramified over three  points $\{0,1, \infty\}$.

In the lemma below we shall prove a stronger result. 
Namely, let $\mathfrak{K}$ be a category of the Galois extensions of the field $\mathbf{Q}$,
such that the  morphisms in $\mathfrak{K}$ are inclusions $K\subseteq K'$, where $K,K'\in\mathfrak{K}$.  
 Likewise, let $\mathfrak{R}$ be a category of the Riemann surfaces,
such that the  morphisms in $\mathfrak{R}$ are holomorphic maps $\mathscr{R}\to \mathscr{R}'$, 
where $\mathscr{R}, \mathscr{R}'\in\mathfrak{R}$.  
Let $F: \mathfrak{K}\to\mathfrak{R}$ be a map acting by the formula $O_K\mapsto\mathscr{R}$,  where
$\mathscr{R}$ is the Riemann surface defined by the isomorphism (\ref{eq3.12}).  
%********************************************************************************************
\begin{remark}\label{rmk3.1}
The category $\mathfrak{R}$ consists of the Riemann surfaces, which are algebraic curves defined over a number field.
In particular,  the morphisms in  $\mathfrak{R}$ can be  defined over the number field.  Both facts follow from  
the property of the AF-algebra $\mathscr{A}_{\mathscr{R}}$ being of a stationary type  \cite[Section 5.2]{N}. 
We refer the reader to Example \ref{ex2.3} and \cite[Lemma 3.1]{Nik2}. 
\end{remark}
%******************************************************************************************
%********************************************************************************************
\begin{lemma}\label{lm3.1}
The map $F: \mathfrak{K}\to\mathfrak{R}$ is  a covariant functor, i.e. $F$  transforms  inclusions 
in the category $\mathfrak{K}$  to  holomorphic maps in the category $\mathfrak{R}$. 
\end{lemma}
%***************************************************************************************
\begin{proof}
Let $K\in\mathfrak{K}$ be a number field and let $\mathscr{R}=F(K)$ be the corresponding
Riemann surface $\mathscr{R}\in\mathfrak{R}$. Let $K\subseteq K'$ be an inclusion, where 
$K'\in \mathfrak{K}$.

Using isomorphism (\ref{eq3.13}),  one gets  an inclusion of the corresponding Serre $C^*$-algebras:  
%**********************************************************************************************
\begin{equation}
\mathscr{A}_{\mathscr{R}}\subseteq \mathscr{A}_{\mathscr{R}'}. 
\end{equation}
%*********************************************************************************************

On the other hand, it is known the algebra $\mathscr{A}_{\mathscr{R}}$
is a coordinate ring of the Riemann surface $\mathscr{R}$ \cite[Theorem 5.2.1]{N}. 
 In particular,  if $h: \mathscr{A}_{\mathscr{R}'}\to \mathscr{A}_{\mathscr{R}}$ is a homomorphism,
  one gets a holomorphic map $w: \mathscr{R}'\to\mathscr{R}$ defined by  a commutative  diagram in Figure 1. 
%*******************************************************************
\begin{figure}[h]
%*******************************************************************
\begin{picture}(300,110)(-70,-5)
\put(20,70){\vector(0,-1){35}}
\put(130,70){\vector(0,-1){35}}
\put(45,23){\vector(1,0){60}}
\put(45,83){\vector(1,0){60}}
\put(13,20){$ \mathscr{A}_{\mathscr{R}'}$}
\put(75,30){$h$}
\put(75,90){$w$}
\put(123,20){$ \mathscr{A}_{\mathscr{R}}$}
\put(17,80){$\mathscr{R}'$}
\put(122,80){$\mathscr{R}$}
\end{picture}
%***********************************************************
\caption{Holomorphic map $w$.}
\end{figure}
%*******************************************************************

Thus $F$ is a functor, which maps the inclusion   $K\subseteq K'$
into a holomorphic map $w: \mathscr{R}'\to\mathscr{R}$. The reader 
can verify that $F$ is a covariant functor. Lemma \ref{lm3.1} is proved. 
\end{proof}

%********************************************************************************************
\begin{lemma}\label{lm3.3}
 The inclusion $\mathbf{Z}\subset O_K$ defines a covering $\mathscr{R}\to \mathbf{C}P^1$
ramified over three  points $\{0,1, \infty\}$.
\end{lemma}
%***************************************************************************************
\begin{proof}
Let $\mathscr{U}$ be the Riemann sphere $\mathbf{C}P^1$ without three points,
which we always assume to be $\{0,1,\infty\}$ after a proper M\"obius transformation. 
It is easy to see, that the fundamental group $\pi_1(\mathscr{U})\cong \mathfrak{F}_2$,  where 
 $\mathfrak{F}_2$ is a free group on two generators $u$ and $v$.

 Since the the Riemann surface $\mathscr{U}$ corresponds to 
  an unlink $\mathscr{L}\cong S^1\cup S^1$,    one gets an isomorphism:
 %***************************************************************************
 \begin{equation}\label{eq3.16}
 B(P^1(\mathscr{U}, \mathcal{L}, \sigma))\cong R[\mathfrak{F}_2].
 \end{equation}
 %**********************************************************************

\smallskip
Consider a two-sided ideal $\mathscr{I}\subset  B(P^1(\mathscr{U}, \mathcal{L}, \sigma))$
generated by relations (\ref{eq3.2}). In view of  (\ref{eq3.16}), we have:
 %***************************************************************************
 \begin{equation}\label{eq3.17}
 B(P^1(\mathscr{U}, \mathcal{L}, \sigma))/\mathscr{I}\cong R[PSL_2(\mathbf{Z})].
 \end{equation}
 %**********************************************************************

\smallskip
In other words, one gets  a homomorphism between the $C^*$-algebras:
%***************************************************************************
 \begin{equation}\label{eq3.18}
 \mathscr{A}_{\mathscr{U}}\to \mathscr{A}_{\mathbf{C}P^1}. 
 \end{equation}
 %**********************************************************************

Using the commutative diagram in Figure 1, we get a holomorphic map between 
the corresponding Riemann surfaces:
%***************************************************************************
 \begin{equation}\label{eq3.19}
 \mathscr{U}\to \mathbf{C}P^1.  
 \end{equation}
 %**********************************************************************

 \bigskip 
 Let now $\mathbf{Z}\subset O_K$ be an inclusion, where $K$ is a number field. 
 By item (ii) of theorem \ref{thm1.3} there exists a Riemann surface $\mathscr{R}\in\mathfrak{R}$ 
 corresponding to $O_K$.  By lemma \ref{lm3.1},  there exists a holomorphic map:
%***************************************************************************
 \begin{equation}\label{eq3.20}
  \mathscr{R}\to \mathbf{C}P^1. 
 \end{equation}
 %**********************************************************************
 
Using (\ref{eq3.19}) and (\ref{eq3.20}), one gets a commutative digram in Figure 2. 
%*******************************************************************
\begin{figure}[h]
%*******************************************************************
\begin{picture}(300,90)(-90,10)
\put(20,70){\vector(0,-1){35}}
%\put(130,70){\vector(0,-1){35}}
\put(45,30){\vector(3,2){60}}
\put(45,83){\vector(1,0){60}}
\put(13,20){$\mathscr{U}$}
\put(17,80){$\mathscr{R}$}
\put(122,80){$\mathbf{C}P^1$}
\end{picture}
%***********************************************************
\caption{The map $\mathscr{R}\to \mathscr{U}$.}
\end{figure}
%*******************************************************************

  \medskip
We use the  diagram in Figure 2  to define a holomorphic map:
 %***************************************************************************
 \begin{equation}\label{eq4.4}
  \mathscr{R}\to \mathscr{U}.
 \end{equation}
 %**********************************************************************
  
  \smallskip
  Since $\mathscr{U}=\mathbf{C}P^1\backslash\{0,1,\infty\}$, one gets the conclusion of 
  lemma \ref{lm3.3}.   
 \end{proof}

\bigskip
Item (iii)  of theorem \ref{thm1.3}  follows from lemmas \ref{lm3.1} and \ref{lm3.3}.

\bigskip
Theorem \ref{thm1.3} is proved.

%**************************************************************************
\section{Belyi's Theorem}
%***************************************************************************
Belyi's Theorem says that the algebraic curve $\mathscr{R}$ can be defined over 
a number field  if and only if there exist a covering $\mathscr{R}\to \mathbf{C}P^1$
ramified over three points of  the Riemann sphere 
$\mathbf{C}P^1$.  This remarkable result was proved by [Belyi 1979] \cite[Theorem 4]{Bel1}. 
In this section we show that Belyi's Theorem follows from  theorem \ref{thm1.3} and remark \ref{rmk3.1}. 
%****************************************************************************
\begin{theorem}\label{thm4.1}
\textbf{(Belyi's Theorem)}
A complete non-singular algebraic curve over $\mathbf{C}$
 can be defined over an algebraic number field if and only if such a curve is a covering
of the Riemann sphere $\mathbf{C}P^1$ ramified over three points. 
\end{theorem}
%*******************************************************************************
\begin{proof}
We identify the Riemann surface $\mathscr{R}\in\mathfrak{R}$ with a complete non-singular algebraic 
curve over  the field of characteristic zero (Chow's Theorem). 

\smallskip
In view of the remark \ref{rmk3.1},  we have $\mathscr{R}\in\mathfrak{R}$  is the algebraic curve
 defined  over a finite extension of the field $\mathbf{Q}$.  
 On the other hand, item (iii) of theorem \ref{thm1.3}  says  that each  Riemann surface 
$\mathscr{R}\in\mathfrak{R}$ is a covering of the  $\mathbf{C}P^1$ ramified over the points
$\{0,1,\infty\}$.  The ``only if'' part  of Belyi's Theorem follows. 
 
\smallskip 
Let $\mathscr{R}$ be a covering of the  $\mathbf{C}P^1$ ramified over the points
$\{0,1,\infty\}$. Using lemma \ref{lm3.1},  one can construct a ring $O_K$ corresponding
to the Riemann surface $\mathscr{R}$.  By item (ii) of theorem \ref{thm1.3} and remark \ref{rmk3.1} we have
$\mathscr{R}\in\mathfrak{R}$.   In other words,  $\mathscr{R}$ is an algebraic curve defined over an algebraic number field.   
The ``if'' part of of Belyi's Theorem is proved.
 \end{proof}
%***********************************************************************
%***********************************************************************
\begin{remark}
It is  interesting  to calculate the  ramification data and  equations of  the Belyi curves $\mathscr{R}$
in terms of the  orders $\Lambda\subseteq O_K$ and number fields $K$ obtained in  theorem \ref{thm1.3}. 
% However this problem is out of the scope of present note. 
\end{remark}
%*******************************************************************

\bibliographystyle{amsplain}

%**********************************************************

\end{document}